\documentclass[a4paper,11pt]{amsart}
\usepackage{amsfonts,amssymb,amsmath,amsthm,latexsym}
\usepackage[usenames]{color}
\usepackage{inputenc}
\usepackage{enumerate}
\usepackage{tikz}
\usetikzlibrary{shapes}
\usetikzlibrary{plotmarks}
\usepackage{mathtools}
\usepackage{float} 

\newtheorem{teor}{Theorem}[section]

\newtheorem{obs}[teor]{Remark}

\newtheorem{prop}[teor]{Proposition}
\newtheorem{remark}[teor]{Remark}

\newtheorem{coro}[teor]{Corollary}
\newcommand{\N}{\mathbb{N}}
\newcommand{\C}{\mathbb{C}}
\newcommand{\Zer}{\hbox{\rm Zer}}
\newcommand{\ord}{\hbox{\rm ord}}
\newcommand{\supp}{\hbox{\rm supp}}

   \usepackage[pagebackref]{hyperref}
  \hypersetup{
colorlinks=true,
linkcolor=cyan,
citecolor=cyan
}


\begin{document}
\title{ Topological type of discriminants of some special families}

\author{Evelia R. Garc\'{i}a Barroso}
\address[Evelia R. Garc\'{i}a Barroso]{Dpto. Matem\'{a}ticas, Estad\'{\i}stica e I.O. Secci\'on de Matem\'aticas, Universidad de La Laguna. Apartado de Correos 456\\
38200 La Laguna, Tenerife, Espa\~na.}
\email{ergarcia@ull.es}

\author{M. Fernando Hern\'andez Iglesias}
\address[M. Fernando Hern\'andez Iglesias]{Dpto. Ciencias - Secci\'{o}n Matem\'{a}ticas, Pontificia Universidad Cat\'{o}lica del Per\'{u}, Av. Universitaria 1801,
San Miguel, Lima 32, Peru}
\email{mhernandezi@pucp.pe}

\thanks{The first-named author was partially supported by the Spanish Project MTM 2016-80659-P}

\date{\today}

\begin{abstract}
We will  describe the topological type of the discriminant curve  of the morphism $(\ell, f)$, where $\ell$ is a smooth curve and $f$ is an irreducible curve  (branch) of multiplicity less than five or a branch that the difference between its Milnor number  and Tjurina number is less than 3. We prove that for a branch of these families, the topological type of the discriminant curve is determined by the semigroup, the Zariski invariant and at most two other analytical invariants of the branch.
\end{abstract}

\maketitle
\tableofcontents


\section{Introduction}

Let $f(x,y)\in  \mathbb C\{x,y\}$ irreducible. The germ of irreducible analytic curve ({\bf branch}) of equation $f(x,y)=0$ is denoted by $C\equiv f(x,y)=0$. Observe that the curves $f(x,y)=0$ and $u(x,y)f(x,y)=0$ are the same, for any unit $u(x,y)\in  \mathbb C\{x,y\}$. The {\bf multiplicity} of $C$, denoted by $m(C)$, is by definition the order of the power series $f(x,y)$. Suppose that $C$ has multiplicity $n$. We will say that $C$ is {\bf singular} if $n>1$. Otherwise $C$ is a {\bf smooth} curve. The {\bf initial form} of $f(x,y)$ is the sum of all terms of $f(x,y)$ of degree $n$. Since $f$ is irreducible its initial form is a power of a linear form. After a linear change of coordinates, if necessary, we can suppose that the initial form of $f(x,y)$ is $y^n$.  Suppose that $C$ has multiplicity $n>1$. We denote by $\mathbb N^*$ the set of positive integers. By Newton's theorem (\cite[Theorem 3.8]{Hefez}) there is $\alpha(x^{1/n}) \in \mathbb{C}\{x\}^{*}=\bigcup_{m\in \mathbb N^*}\mathbb{C}\{x^{1/m}\}$ with $\alpha(0)=0$ such that $f(x,\alpha(x^{1/n}))=0$ and we say that $\alpha(x^{1/n})  \in \mathbb{C}\{x\}^{*}$ is a {\bf Newton-Puiseux root} of $C$. Let us denote by $\Zer(f)$ the set of Newton-Puiseux roots of $C$. Let $\alpha(x^{1/n})  \in \Zer(f)$. After Puiseux theorem (\cite[Corollary 3.12]{Hefez}) we have that $\Zer(f)=\left\{\alpha_j:=\alpha(\omega^j x^{1/n})\right\}_{j=1}^n$, where $\omega$ is a $n$th-primitive root of the unity. Hence
\begin{equation}
\label{ecuaci—n}
f(x,y)=u(x,y)\displaystyle\prod_{j=1}^{n}\Big(y-(\alpha(\omega^{j}x^{1/n}))\Big),
\end{equation}
where $u \in \mathbb{C}\{x,y\}$ is a unit. After a change of coordinates, if necessary, we can write $\alpha(x)=\sum_{i\geq s_1}a_i x^{i/n}$, where $s_1>n$ and $s_1\not\equiv 0$ mod $n$.

\medskip

If we put $x=t^{n}$, where $t$  is a new variable, the Newton-Puiseux root $\alpha(x^{1/n})$ can be written as
\[
\left\{ \begin{array}{rll}
x(t) &=& t^{n}\\
y(t) &=& \displaystyle\sum_{i\geq s_1}a_{i}t^{i},
\end{array}
\right.
\]
what we will call {\bf\emph{Puiseux parametrisation }}of $C$. 

\medskip

\noindent There are  $g\in \mathbb N$ and a sequence $\left(\beta_{0}=n<\beta_1=s_1<\beta_2<\cdots<\beta_{g}\right)$ of nonnegative integers  such that 
\begin{equation}\label{lisas}
\{\ord(\alpha_i-\alpha_j)\;:\; \alpha_i,\alpha_j\in \Zer(f),\;i\neq j\}=\left\{\frac{\beta_l}{\beta_0}\;:\;1\leq l\leq g\right\}\subseteq \mathbb Q\backslash  \mathbb Z.
\end{equation} 

\noindent The sequence $(\beta_{0},\cdots,\beta_{g}) \subseteq \mathbb{N}$ is called sequence of {\bf \emph{characteristic exponents}}  of $C$. The number $g$ is a topological invariant called {\bf genus} of the branch $C$.\\

\noindent Consider the set
\[
S(C):=\{i_0(f,h)\,:\, h\in \C\{x,y\},\,\,h\not\equiv 0\, \hbox{\rm mod } f\},
\]

\noindent where $i_0(f,h)=\dim_{\C}\C\{x,y\}/(f,h)$ is the {\bf intersection number} (or {\bf intersection multiplicity}) of $f(x,y)=0$ and $h(x,y)=0$ at the origin. It is well-known that $S(C)$ is a semigroup called {\bf semigroup of values} of the branch $C$. The complementary of $S(C)$ in $\N$ is finite. The {\bf conductor} of $S(C)$ is by definition the greatest natural number $c\in \N$ such that for every natural number  $N\in \N,$ with $N\geq c$,  is an element of $S(C)$.

\medskip

\noindent The semigroup $S(C)$ admits a minimal system of generators $(s_0,s_1,\ldots,s_g)$, where  $s_{i-1}<s_{i}$, $g$ is the genus of $C$, $s_0=n=i_0(f,x)$ and
$s_1=m=:i_0(f,y)$. It is a well-known property of $S(C)$ (\cite[page 88, inequality (6.5)]{Hefez}) that $e_k:=\gcd(s_0,\ldots,s_k)=\gcd(\beta_0,\ldots,\beta_k)$ for $0\leq k\leq g$ and $e_{k-1}s_k<e_ks_{k+1}$ for $1\leq k\leq g-1$.
 
\medskip 
If $n>2$ we  have $c\geq s_1+1$. Let $q$ be the number of natural numbers between $s_1$ and $c$ which are not in $S(C)$. We can verify (see \cite[page 21]{Zariski}) that $q=\frac{c}{2}-s_1+\left[\frac{s_1}{s_0}\right]+1$, for  $s_0=n>2$, where $\left[ z \right]$ denotes the integral part of $z\in \mathbb R$.

\medskip

Let $f,h\in  \C\{x,y\}$ be irreducible power series. After Halphen-Zeuthen formula we get
\begin{equation}
\label{Halphen}
i_0(f,h)=\sum_{i,j}\ord (\gamma_j-\alpha_i),
\end{equation}
where $\Zer f=\{\alpha_i\}_i$ y $\Zer h=\{\gamma_j\}_j.$ 

\medskip

\noindent Two branches $C$ and $D$ have the {\bf same topological type} (or they are {\bf equisingular}) if they are topologically equivalent as embedded surfaces in $\C^2$. It is well-known (\cite[Chapter II]{Zariski}) that two branches are equisingular if and only if they have the same semigroup of values or equivalently they have the same characteristic exponents. Denote by ${\mathcal E}(C)$ the set of branches which are equisingular to $C$. In the set
${\mathcal E}(C)$ we define the next equivalence relation: two branches $D_1$ and $D_2$ in ${\mathcal E}(C)$ are {\bf analytically equivalent}, and we will denote it by $D_1\cong D_2$ if there exists an analytic isomorphism $T:U_1\longrightarrow U_2$ such that $U_i$ are neighbourhoods of the origin, $D_i$ is defined in $U_i$, $1\leq i\leq 2$ and $T(D_1\cap U_1)=D_2\cap U_2$. The {\bf moduli space} of the equisingularity class ${\mathcal E}(C)$ is the quotient space ${\mathcal E}(C)/\cong$. Let $\nu_1<\nu_2<\ldots<\nu_q$ be the integers of the set $\{s_1+1,\ldots,c-1\}$ which are not in $S(C)$. Zariski proved \cite[Proposition 1.2, Chapter III]{Zariski} that there exists a branch $\overline C$ analytically equivalent to $C$, parametrized as follows:

\begin{equation}
\label{short-paramet}
\left\{ \begin{array}{l} \bar x=t^n\\ \bar y=t^{s_1}+\sum_{i=1}^qa_{i}t^{\nu_i}. \end{array} \right.
\end{equation}

Put $\Omega:=\{\omega=g(x,y)dx+h(x,y)dy\,:\,g,h\in \C\{x,y\}\}$. If $(x(t),y(t))$ is a Puiseux parametrisation of $C$  we put

\[
\nu(\omega):=\ord \left( f(x(t),y(t))x'(t)+ g(x(t),y(t))y'(t)\right)+1.
\]

\smallskip

Let $\Lambda:=\{\nu(\omega)\,:\,\omega \in \Omega\}$. If $\Lambda \backslash S(C)\neq \emptyset$ then the number $\lambda:=\min \left( \Lambda \backslash S(C)\right)-v_0$ is an analytical invariant of $C$ called {\bf Zariski invariant}.

\medskip

After \cite[Lemma 2.6, Chapter IV]{Zariski} we can rewrite the parametrization (\ref{short-paramet}) in the next form:

\[
\left\{ \begin{array}{l} \bar x=t^n\\ \bar y=t^{s_1}+at^{\lambda}+\hbox{\rm a finite sum of terms  $a_{i}t^{\nu_i}$}, \end{array} \right.
\]

\noindent where $a\neq 0$, $\nu_i>\lambda>s_1$.

\medskip

Let $f(x,y)=\sum_{i,j}a_{ij}x^iy^y\in \mathbb C\{x,y\}$. The {\bf support} of $f$ is $\supp(f):=\{(i,j)\in \mathbb N^2\;:\; a_{ij}\neq 0\}$. The {\bf Newton polygon} of $f$, denoted by ${\mathcal N}(f)$, is by definition the convex hull of $\supp(f)+\mathbb R_{\geq 0}^2$. Observe that ${\mathcal N}(f)={\mathcal N}(uf)$ for any unit $u\in 
\mathbb C\{x,y\}$. Nevertheless the Newton polygon depends on coordinates. The {\bf  inclination} of any compact face $L$ of ${\mathcal N}(f)$ is by definition de quotient of the length of the proyection of $L$ over the horizonal  axis by the length of its projection over the vertical axis. The Newton polygon of $f$ gives information on the Newton-Puiseux roots of $f(x,y)=0$. More precisely, if $L$ is a compact face of ${\mathcal N}(f)$ of inclination \mdseries{i} and the length of its projection over the vertical axis is $\ell_2$ then $f$ has $\ell_2$ Newton-Puiseux roots of order \mdseries{i} (see \cite[Lemme 8.4.2]{Chenciner}).
\medskip

\begin{center}
\begin{tikzpicture}[x=0.7cm,y=0.7cm]
\tikzstyle{every node}=[font=\small]
\draw[->] (0,0) -- (5,0) node[right,below] {$\;$};
\draw[->] (0,0) -- (0,4) node[above,left] {$\;$};
\draw[thick] (2.5,2.7) node[left] {$L$};
\draw[thick] (4.8,2) node[left] {\mdseries i=$\frac{\ell_1}{\ell_2}$};
\node[draw,circle,inner sep=1pt,fill] at (1,3) {};
\node[draw,circle,inner sep=1pt,fill] at (4,1) {};
\draw[dashed] (0,3) --(1,3);
\draw[dashed] (0,1) --(4,1);
\draw[dashed] (4,0) --(4,1);
\draw[dashed] (1,0) --(1,3);
\node[draw,circle,inner sep=1pt,fill] at (4,1) {};
\draw[ultra thick] (1,3) -- (4,1);
\draw[<->] (-0.5,3) to [bend right] (-0.5,1);
\draw[-] (-1.4,2.5) node[right,below] {$\ell_2$};
\draw[<->] (1,-0.5) to [bend right] (4,-0.5);
\draw[-] (2.7,-0.9) node[below] {$\ell_1$};
\end{tikzpicture}
\end{center}

We say that $f(x,y)\in \mathbb C\{x,y\}$ is {\bf non degenerate} in the sense of Kouchnirenko, with respect to the coordinates $(x,y)$, if for any compact edge $L$ of ${\mathcal N}(f)$ the polynomial $f_L(x,y):=\sum_{(i,j)\in L\cap \supp(f)}a_{ij}x^iy^j$ does not have critical points outside the axes $x=0$ and $y=0$, or equivalently, the polynomial $F_L(z):=\frac{f_L(1,z)}{z^{j_0}}$ has no multiple roots, where $j_0:=\min \{j\in \mathbb N\;:\; (i,j)\in L\}$. Since ${\mathcal N}(f)={\mathcal N}(uf)$, for any unit $u\in 
\mathbb C\{x,y\}$, the notion of non degeneracy is extended to curves. The topological type of non degenerate plane curves are completely determined by their Newton polygons (see \cite[Proposition 4.7]{Oka} and \cite[Theorem 3.2]{GB-L-P2007}). 

\medskip

\noindent Let $\ell(x,y)=0$ be a 
smooth curve and $f(x,y)=0$ defining an isolated singularity at $0\in \mathbb C^2$.  Assume that $\ell(x,y)$ does not divide $f(x,y)$ and consider the morphism 

\begin{equation}
\label{disc-m}
\begin{array}{rll}
(\ell ,f)\colon (\mathbb C^{2},0)&\longrightarrow& (\mathbb C^2,0)\\
(x,y)&\longrightarrow& (u,v):=(\ell(x,y),f(x,y)).
\end{array}
\end{equation}

\noindent There are two curves associated with $(\ell ,f)$:  the {\bf polar curve} 
$\frac{\partial \ell}{\partial x}\frac{\partial f}{\partial
y} - \frac{\partial \ell}{\partial y}\frac{\partial f}{\partial x}=0$ and its direct image $D(u,v)=0$ 
which is called the {\bf discriminant curve} of the morphism $(\ell ,f)$.

\medskip

The topological type of the polar curve depends on the analytical type of $\ell(x,y)=0$ and $f(x,y)=0$. In \cite{Hefez-Hernandes-HI2017} the authors  completely determine the topological type of the generic polar curve when the multiplicity of $f(x,y)=0$ is less than five.

\medskip

The Newton polygon of $D(u,v)$ in the coordinates $(u,v)$ is called {\bf jacobian Newton polygon} of the morphism $(\ell,f)$. This notion was introduced by Teissier in \cite{Teissier}, who proved that the inclinations of this jacobian polygon are topological invariants of $(\ell,f)$ called {\bf polar invariants}. After Merle \cite{Merle}, when $f$ is irreducible with semigroup of values $S(f)=\langle s_0,s_1,\ldots,s_g\rangle$ then the jacobian Newton  polygon of $(\ell,f)$ has $g$ compact edges $\{E_i\}_{i=1}^g$. The length of the projection of $E_i$ on the vertical axis is $\left(\frac{e_{i-1}}{e_i}-1\right)\cdot \frac{e_{i-1}}{e_0}$. The length of the projection of $E_i$ on the horizontal axis is $\left(\frac{e_{i-1}}{e_i}-1\right)\cdot s_i$.  Hence the inclinations (quotient between the  length of the horizontal projection and the length of the vertical projection) of the compact edges of the jacobian polygon are $s_1<\frac{e_1}{e_0}s_2<\frac{e_2}{e_0}s_3<\cdots<\frac{e_{g-1}}{e_0}s_g$

\begin{center}
\begin{tikzpicture}[x=0.7cm,y=0.7cm]
\tikzstyle{every node}=[font=\small]
\draw[->] (0,0) -- (5,0) node[right,below] { $\;$};
\draw[->] (0,0) -- (0,4) node[above,left] {$\;$};
\draw[thick] (3.6,2) node[left] {$E_i$};
\node[draw,circle,inner sep=1pt,fill] at (1,3) {};
\node[draw,circle,inner sep=1pt,fill] at (4,1) {};
\draw[dashed] (0,3) --(1,3);
\draw[dashed] (0,1) --(4,1);
\draw[dashed] (4,0) --(4,1);
\draw[dashed] (1,0) --(1,3);
\node[draw,circle,inner sep=1pt,fill] at (4,1) {};
\draw[ultra thick] (1,3) -- (4,1);
\draw[<->] (-0.5,3) to [bend right] (-0.5,1);
\draw[-] (-2.8,2.5) node[right,below] {$\left(\frac{e_{i-1}}{e_i}-1\right)\cdot \frac{e_0}{e_{i-1}}$};
\draw[<->] (1,-0.5) to [bend right] (4,-0.5);
\draw[-] (2.7,-0.9) node[below] {$\left(\frac{e_{i-1}}{e_i}-1\right)\cdot s_i$};
\end{tikzpicture}
\end{center}

\medskip

In \cite{Hungarica} the authors study the pairs  $(\ell,f)$  for which the discriminant curve is non degenerate in  the Kouchnirenko sense. In particular, when $f$ is irreducible, after \cite[Corollary 4.4]{Hungarica}  the discriminant curve $D(u,v)=0$ is non degenerate if and only if the multiplicity of $f(x,y)=0$ equals two or equals four and genus equals two. Otherwise the discriminant curve is degenerate. Our aim in this paper will be to describe the topological type of the discriminant curve $D(u,v)=0$ of the morphism $(\ell, f),$ where $f$ is irreducible and belonging  to some special families, as for example, branches of multiplicity less than five, branches $C$ such that the difference between its {\bf Milnor number} $\mu(C)$ and {\bf Tjurina number} $\tau (C)$ is less than 3, with 
$\mu(C)=i_0\left(\frac{\partial f}{\partial x},\frac{\partial f}{\partial y}\right)=\hbox{\rm dim}_{\mathbb C}\mathbb C\{x,y\}/\left(\frac{\partial f}{\partial x},\frac{\partial f}{\partial y}\right)$ and 
$\tau (C)=\hbox{\rm dim}_{\mathbb C}\mathbb C\{x,y\}/\left(f,\frac{\partial f}{\partial x},\frac{\partial f}{\partial y}\right)$; where $\left(\frac{\partial f}{\partial x},\frac{\partial f}{\partial y}\right)$ 
(resp.$\left(\frac{\partial f}{\partial x},\frac{\partial f}{\partial y},f\right)$) denotes the ideal of $\C\{x,y\}$ generated by $\frac{\partial f}{\partial x}$ and $\frac{\partial f}{\partial y}$ (resp. by $\frac{\partial f}{\partial x}$, $\frac{\partial f}{\partial y}$ and $f$).

\medskip

For these families of plane branches we determine the topological type of their discriminant curves, in the spirit of \cite{Hefez-Hernandes-HI2017}. We prove that the topological type of the discriminant curve $D(u,v)=0$ is determined, at most, by the semigroup of values $S(f)$, the Zariski invariant and  two other analytical invariants of the curve $f(x,y)=0$. In all cases we explicitly determine such analytical invariants. 
Hence, in order to describe the topological type of the discriminant curve of a branch, it is necessary the same number of analytical invariants of the initial branch, as it happens for its generic polar curves (see \cite{Hefez-Hernandes-HI2017}).
Finally in Section \ref{tablas} we summarize the different topological types of the discriminant curve in some tables.

\section{Equation of the discriminant curve}

\noindent An analytic change of coordinates does not affect the discriminant curve of the morphism defined in  (\ref{disc-m}) (see for example \cite[Section 3]{Casas-Asian}). Hence in what follows  we assume that $\ell(x,y)=x$. Then $\frac{\partial f}{\partial y}=0$ is the polar curve of the morphism $(x,f)$. 

\medskip

\noindent In this paper we  will determine the topological type of the discriminant curve of the morphism (\ref{disc-m}) for $\ell(x,y)=x$ and $f(x,y)\in \mathbb C\{x,y\}$ irreducible belonging to some special families. The corresponding study relative to the polar curves was done
in  \cite{Hefez-Hernandes-HI2018} and \cite{tesis Fernando}, where  the authors  characterize the equisingularity classes of irreducible plane curve germs whose general members have non degenerate general polar curves. In addition, they give explicit Zariski open sets of curves
in these equisingularity classes whose general polars are non degenerate and describe their topology.

\medskip

\medskip
  
\noindent Suppose that the Newton-Puiseux factorizations of $f(x,y)$ and $\frac{\partial f}{\partial y}(x,y)$ are of the form
\begin{equation}
\label{ppp1}
f(x,y)= u_1(x,y)\prod_{i=1}^n [y-\alpha_i(x)], 
\end{equation}
\begin{equation}
  \label{ppp2}
\frac{\partial f}{\partial y}(x,y)=u_2(x,y)\prod_{j=1}^{n-1} [y-\gamma_j(x)],
\end{equation}
where 
$u_1(x,y),u_2(x,y) \in \mathbb C\{x,y\}$ are units, $n=\ord f$,  $\Zer(f)=\{\alpha_i(x)\}_i$ and $\Zer\left(\frac{\partial f}{\partial y}\right)=\{\gamma_j(x)\}_j$. If  $f$ is irreducible of order $n$ then $n$ is the smallest natural number such that $\{\alpha_i(x)\}_i\subset \mathbb C\{x^{1/n}\}$. Moreover if we fix $\alpha_i(x^{1/n})$ then $\alpha_j(x^{1/n})=\alpha_i(\omega x^{1/n})$ for any $1\leq j\leq n$, where $w$ is a $n$th-root of the unity.

\medskip

\noindent Following \cite[Lemma 5.4]{GB-G} the discriminant curve of the morphism $(x,f)$ can be written as
\begin{equation}
\label{disc-e}
D(u,v) = \prod_{j=1}^{n-1}(v-f(u,\gamma_j(u))) .
\end{equation}

\section{Discriminants of branches of small multiplicities}

\noindent In this section we  determine the topological type of the discriminant of the morphism
given in (\ref{disc-m}), where $C\equiv f(x,y)=0$ has small multiplicity.  For this we will make use  the results of \cite{Zariski} and the analytic classification of plane branches of multiplicity less than or equal to four, given in \cite{Hefez-Hernandes-2009}.

\subsection{Discriminants of branches of multiplicity $2$}
Let $C\equiv f(x,y)=0$ be  a branch of multiplicity $2$. The minimal system of generators of the semigroup of $C $ is $\langle2,s_1\rangle$, where $s_1$ is an odd natural number. By \cite[Chapitre V]{Zariski} the moduli space of branches of multiplicity two have a unique point which parametrization is given by $(t^2,t^{s_1})$, that is the branch $y^2-x^{s_1}=0$. Then $f_y(x,y)=2y$ whose Newton-Puiseux root is  $y = 0$. Hence, after (\ref{disc-e}) we have $D(u,v)=v-f(u,0)=v+u^{s_1}$, that is the discriminant curve is smooth. The Newton polygon of $D(u,v)$ has only one compact edge. The univariate polynomial associated with this edge is $z+1$, so $D(u,v)$ is non degenerate.

\subsection{Discriminants of branches of multiplicity $3$}
Let $C\equiv f(x,y)=0$ be  a branch of multiplicity $3$. The minimal system of generators of the semigroup of $C $ is $\langle3,s_1\rangle$, where $s_1\in \mathbb N$ such that $s_1\not\equiv 0$ mod $3$. By \cite[Chapitre V]{Zariski} the moduli space of branches of multiplicity three is completely determined by the semigroup of the branch and its Zariski invariant $\lambda$. The corresponding  normal forms are :\\

\[x=t^3,\;\;y=t^{s_1}+t^{\lambda},
\]

\medskip

\noindent where  $\lambda=0$ or  if $\lambda\neq 0$ we have 
\begin{equation}
\label{3}
\lambda=\left\{\begin{array}{ll} 3e+3k+4 \; &\hbox{\rm when }s_1=3e+2 \\ 3e+3k+2 \; &\hbox{\rm when }s_1=3e+1,
\end{array}
\right. \hbox{where $0\leq k\leq e-2$.}
\end{equation}

\medskip

\noindent Observe that if  $\lambda\neq 0$ then $\frac{s_1+\lambda}{2}$ is a natural number greater than $2$.

\medskip

\begin{prop}
Let $C\equiv f(x,y)=0$ be  a branch of semigroup $\langle 3,s_1\rangle$ and Zariski invariant equals $\lambda$. The discriminant curve $D(u,v)=0$ is degenerate and its topological type is determined by $(3,s_1,\lambda)$ in the next way:
\begin{enumerate}
\item If $\lambda=0$ then the discriminant is the double smooth branch $(v+u^{s_1})^2=0$.
\item If $\lambda\neq 0$ then
\begin{enumerate}
\item when $\gcd(2,s_1+\lambda)=2$, the discriminant curve is the union of two smooth branches $D_i(u,v)=0$, $1\leq i\leq 2$, with intersection number $i_0(D_1,D_2)=\frac{s_1+\lambda}{2}$.\item When $\gcd(2,s_1+\lambda)=1$, the discriminant is a branch of semigroup $\langle2,s_1+\lambda \rangle$.
\end{enumerate}
\end{enumerate}
\end{prop}

\noindent \begin{proof}
\noindent Suppose $\lambda=0$. The implicit equation of the normal form is $f(x,y)=y^3-x^{s_1}=0$. Then $f_y(x,y)=3y^2$ whose Newton-Puiseux root is $y = 0$, with multiplicity two. Hence, after (\ref{disc-e}) we have $D(u,v)=(v-f(u,0))^2=(v+u^{s_1})^2$, that is the discriminant is a double smooth branch.

\medskip

Suppose now $\lambda\neq 0$. The normal forms are $x=t^3,\,\,y=t^{s_1} + t^{\lambda},$ with $\lambda$ as in (\ref{3}). After (\ref{ecuaci—n}) the  implicit equation  is $f(x,y)=y^3-3x^{\frac{s_1+\lambda}{3}}y-(x^{s_1} + x^{\lambda})$. Then $f_y(x,y)=3y^2-3x^{\frac{s_1+\lambda}{3}}$, whose Newton-Puiseux roots are $\gamma_i(x)=\pm x^{\frac{s_1+\lambda}{6}}$ for $1\leq i\leq 2$. Using (\ref{disc-e}) we have $D(u,v)=\left(v+u^{s_1}-2u^{\frac{s_1+\lambda}{2}}+u^{\lambda}\right)\left(v+u^{s_1}+2u^{\frac{s_1+\lambda}{2}}+u^{\lambda}\right) $. Hence $\Zer(D)=\{\eta_1:=-u^{s_1}-2u^{\frac{s_1+\lambda}{2}}-u^{\lambda},\eta_2:=-u^{s_1}+2u^{\frac{s_1+\lambda}{2}}-u^{\lambda}\}$ and 
$\ord(\eta_1-\eta_2)=\frac{s_1+\lambda}{2}$.
We distinguish two cases: if $\gcd(2,s_1+\lambda)=2$ then by (\ref{lisas}) we conclude that $D(u,v)=0$ has two smooth branches of equations $D_1(u,v):=v+u^{s_1}-2u^{\frac{s_1+\lambda}{2}}+u^{\lambda}$ and $D_2(u,v)=v+u^{s_1}+2u^{\frac{s_1+\lambda}{2}}+u^{\lambda}$ such that, after Halphen-Zeuthen formula, the intersection multiplicity  is $i_0(D_1,D_2)=\frac{s_1+\lambda}{2}$. If $\gcd(2,s_1+\lambda)=1$ then the discriminant curve $D(u,v)=0$ is a singular branch of semigroup $\langle 2,s_1+\lambda \rangle$.

\medskip

The Newton polygon of $D(u,v)$ is elementary (it has only one compact edge) and the univariate polynomial associated with its compact edge is $(z+1)^2$. Hence the discriminant $D(u,v)=0$ is degenerate.
\end{proof}

\medskip

\begin{coro}
\label{dred}
If $C$ is a branch of multiplicity $2$ or $3$ and non-zero Zariski invariant then the discriminant curve $D(u,v)=0$ has not multiple irreducible branches.
\end{coro}

Corollary 3.2 does not hold for branches of multiplicity $4$ as the proof of Proposition \ref{sigma5} shows.

\subsection{Discriminants of branches of multiplicity $4$}
\label{multiplicidad 4}
Let $C\equiv f(x,y)=0$ be  a branch of multiplicity $4$.  The branch $C$ may have genus one or two. 

\begin{prop}
\label{4NG}
Let $C\equiv f(x,y)=0$ be  a branch of semigroup $\langle 4,s_1, s_2\rangle$ . Then the discriminant curve $D(u,v)=0$ is non-degenerate. Moreover $D(u,v)=D_1(u,v)D_2(u,v)$, where $D_1$ is a smooth branch, $D_2$ is a singular branch of semigroup $\langle 2, s_2\rangle$and the intersection multiplicity between both branches is
$i_0(D_1,D_2)=2s_1$.
\end{prop}
\noindent \begin{proof}
\noindent 
For genus two, and after the second part of \cite[Corollary 4.4]{Hungarica} we get that $D(u,v)=0$ is non degenerate and we can determine its topological type from its Newton polygon (see \cite[Proposition 4.7]{Oka} and \cite[Theorem 3.2]{GB-L-P2007}), which is the jacobian Newton polygon of $(x,f)$ (see Figure \ref{JNP}).

\begin{figure}[h]
\begin{center}
\begin{tikzpicture}[x=0.7cm,y=0.7cm]
\tikzstyle{every node}=[font=\small]
\draw[->] (0,0) -- (15,0) node[right,below] {$i$};
\draw[->] (0,0) -- (0,3) node[above,left] {$j$};
\draw[thick] (0,1.5) node[left] {$3$};
\node[draw,circle,inner sep=1.2pt,fill] at (0,1.5) {};
\draw[thick] (0,1) node[left] {$2$};
\node[draw,circle,inner sep=1.2pt,fill] at (13.5,0) {};
\node[draw,circle,inner sep=1.2pt,fill] at (3,1) {};
\draw[thick] (0,1.5) --(3,1);
\draw[thick] (3,1) --(13.5,0);
\draw[dashed] (3,1) --(0,1);
\draw[dashed] (3,1) --(3,0) node[below] {$s_1$};
\draw (13.5,0) node[below]  {$s_1+s_2$};
\end{tikzpicture}
\end{center}
\caption{Jacobian Newton polygon of a branch with semigroup $\langle 4,s_1, s_2\rangle$.}\label{JNP}
\end{figure}
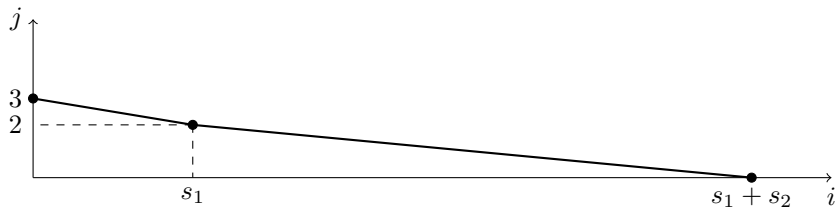 

\medskip

\noindent Since $s_2$ is an odd natural number then  $\gcd(s_1+s_2,2)=1$. Moreover, after the properties of $S(f)$,  we get $e_0s_1=4s_1<2s_2=e_1s_2$, hence $s_2>2s_1$ and $D(u,v)=D_1(u,v)D_2(u,v)$, where $D_1(u,v)=0$ is a smooth curve admitting as parametrization $(t,t^{s_1}+\cdots)$ and $D_2(u,v)=0$ is a singular curve of genus $1$ and semigroup of values equals $\langle 2, s_2\rangle$. Finally, after Halphen-Zeuthen formula, the intersection multiplicity between both branches is
$i_0(D_1,D_2)=2s_1$.
\end{proof}
\medskip

\noindent Suppose now that the branch $C$ has genus $1$ and semigroup   of values equals $\langle 4, s_1\rangle$. By \cite{Hefez-Hernandes-2009}  the moduli space of branches of multiplicity four and genus 1 has five  families of normal forms:

\begin{equation}
\label{sigmas}
\hbox{\rm NF 4.$i$:}\;\;x=t^4,\;\;y_i=\sigma_i(t), \;\;\;\; 1\leq i\leq 5,
\end{equation}

where    $\lambda_i$ is the Zariski invariant of the $i$th-normal form family. More precisely we have

\begin{enumerate}
\item $(\lambda_1,\sigma_1(t))=(0,t^{s_1})$, 
\item If $2\leq i \leq 4$ then $\lambda_i=2s_1-4j$ for $2\leq j \leq [\frac{{s_1}}{4}]$ and

  \begin{eqnarray*}\;\;\;\;\;\;\;\;\;\;\;\; \sigma_2(t)&=&t^{s_1}+t^{\lambda_2}+a_kt^{3{s_1}-(4[\frac{{s_1}}{4}]+j+1-k)}+ \cdots + a_{j-[\frac{{s_1}}{4}]-2}t^{3{s_1}-4([\frac{{s_1}}{4}]+3-k)},\end{eqnarray*}
 with $a_k \neq 0,\,$ and $1\leq k \leq [\frac{{s_1}}{4}]-j$.
				
 \begin{eqnarray*}\sigma_3(t)&=&t^{s_1}+t^{\lambda_3}+	\frac{3{s_1}-4j}{2{s_1}}t^{3{s_1}-8j}+a_{[\frac{{s_1}}{4}]-j+2} t^{3{s_1}-4(2j-1)}+\cdots\\ &+& a_{[\frac{{s_1}}{4}]}t^{3{s_1}-(j+1)}. \end{eqnarray*}

  \begin{eqnarray*}\sigma_4(t)&=&t^{s_1}+t^{\lambda_4}+a_{[\frac{{s_1}}{4}]-j+1}t^{3{s_1}-8j}+a_{[\frac{{s_1}}{4}]-j+2}t^{3{s_1}-4(2j-1)}+ \cdots \\
  &+& a_{[\frac{{s_1}}{4}]-1}t^{3{s_1}-4(j+2)}, \end{eqnarray*}
  where $a_{[\frac{{s_1}}{4}]-j+1} \neq \frac{3{s_1}-4j}{2{s_1}}$.

\item  If $i=5$ then $\lambda_5=3s_1-4j$ for $2\leq j \leq [\frac{{s_1}}{2}]$ and 

\begin{eqnarray*} \sigma_5(t)&=&t^{s_1}+t^{3s_1-4j}+a_kt^{2s_1-4(j-[\frac{s_1}{4}]-k)} + a_{k+s}t^{2s_1-4(j-[\frac{s_1}{4}] k-s)}\\
&+&\cdots  
\end{eqnarray*}

			\end{enumerate}
After the Newton-Puiseux Theorem, the $i$th-normal form admits the equation 
\begin{equation}\label{eq:f}
f_i(x,y)=\prod_{\omega^4=1}(y-\sigma_i(\omega x^{1/4})),
\end{equation}
where $\omega$ is a $4$th primitive root of the unity.

\noindent Hence, we obtain the implicit equation for each normal form family: 
\begin{enumerate}
\item [\hbox{\rm NF 4.$1$:}] $f_1(x,y)=y^4-x^{s_1}$.

\item [\hbox{\rm NF 4.$i$:}]  For $2\leq i \leq 4$ we get 
\begin{equation}
\label{normal form 4}
f_i(x,y)=y^4+P_i(x)y^2+Q_i(x)y+x^{s_1}u(x),
\end{equation}

\noindent where $u(x)\in \mathbb C\{x\}$ is a unit (that is $u(0)\neq 0$), $\lambda_i=2s_1-4j$ for $2\leq j \leq [\frac{{s_1}}{4}]$, $Q_i(x)=-4x^{s_1-j}+\cdots$ and

\[
P_i(x)=\left \{\begin{array}{ll}
-4a_kx^{s_1-\left(\left[\frac{s_1}{4}\right]+j+1-k\right)}+\cdots&\hbox{\rm for } i=2\\
&\\
bx^{s_1-2j}+\cdots &\hbox{\rm for } i=3 \; \hbox{\rm where }b=\frac{-8(s_1-j)}{s_1}.\\
&\\
(-2-4a_{\left[\frac{s_1}{4}\right]-j+1})x^{s_1-2j}+\cdots& \hbox{\rm for } i=4 \; \hbox{\rm and }-2-4a_{\left[\frac{s_1}{4}\right]-j+1} \neq 0\\
&\\
0& \hbox{\rm for } i=4, \;\;-2-4a_{\left[\frac{s_1}{4}\right]-j+1}=0\\
\; & \hbox{\rm and } a_{\left[\frac{s_1}{4}\right]-j+2}= \cdots =a_{\left[\frac{s_1}{4}\right]-1}=0 \\
&\\
-4a_{\left[\frac{s_1}{4}\right]-j+k_0}x^{s_1-(2j-(k_0-1))}+\cdots& \hbox{\rm for } i=4, \;\;-2-4a_{\left[\frac{s_1}{4}\right]-j+1}=0\\
\; &\hbox{\rm and }a_l\neq 0 \; \hbox{\rm for  some }l=\left[\frac{s_1}{4}\right]-j+2,\ldots, \left[\frac{s_1}{4}\right]-1
\end{array}
\right .
\]

\noindent where $ 2\leq k_0 \leq j-1$ such that 
\[
k:= \min \left\{l\;:\;a_l \neq 0 \,\,, \,\,  \left[\frac{s_1}{4}\right]-j+2 \leq l <  \left[\frac{s_1}{4}\right]\right\}=\left[\frac{s_1}{4}\right]-j+k_0.
\]

Observe that 
\begin{equation}
\label{orden P}\ord P_i(x)\geq \frac{2}{3}(s_1-j).
\end{equation}

\item [\hbox{\rm NF 4.$5$:}] $f_5(x,y)=y^4+P(x)y^2+Q(x)y+R(x)$, where

\begin{eqnarray}
\label{P5}
P(x)&=&-4x^{s_1-j}-2a_{k}^2x^{s_1-2(j-[\frac{s_1}{4}]-k)}-4a_ka_{k+s}x^{s_1-2(j-s_1-[\frac{s_1}{4}]-k)+s} \nonumber\\ 
&-&
2a_{k+s}^2 x^{s_1-2(j-[\frac{s_1}{4}]-k-s )}+\cdots
\end{eqnarray}

\begin{eqnarray}
\label{Q5}
\;\;\;\;\;Q(x)&=&-4a_kx^{s_1-j+[\frac{s_1}{4}]+k}-4a_{k+s}x^{s_1-j
	-[\frac{s_1}{4}]+k+s} \nonumber\\ 
&-&4a_kx^{2(s_1-j)-(j-[\frac{s_1}{4}]-k)} - 4a_{k+s}x^{2(s_1-j)-(j-[\frac{s_1}{4}]-k-s)}+\cdots)\;\;\;\;
\end{eqnarray}

\begin{eqnarray}
\label{R5}
R(x)&=&-x^{s_1}+2x^{2(s_1-j)} -x^{3s_1-4j}+a_kx^{2s_1-4(j-[\frac{s_1}{4}]-k)}+\cdots\;\;\;\;
\end{eqnarray}
\medskip
\end{enumerate}	

for $2\leq j \leq [\frac{{s_1}}{2}]$.

\begin{prop}
\label{1-4}
Let $C\equiv f(x,y)=0$ be  a branch belonging to the family NF 4.$i$, for $1\leq i\leq 4$. Then the discriminant curve $D(u,v)=0$ is degenerate and its topological type is determined by the semigroup $S(f)=\langle 4,s_1\rangle$ and the Zariski invariant $\lambda_i$ of $C$. Moreover

\begin{enumerate}
\item If $\lambda_i=0$ then the discriminant curve is the triple smooth branch $(v+u^{s_1})^3=0$.
\item If $\lambda_i\neq 0$ then $\lambda_i=2{s_1}-4j$  for $2\leq j \leq [\frac{{s_1}}{4}]$ and 
\begin{enumerate}
\item when $\gcd(3,2s_1+\lambda_i)=1$ the discriminant curve $D(u,v)=0$ is a branch of semigroup $\langle 3,2s_1+\lambda_i\rangle$;
\item when $\gcd(3,2s_1+\lambda_i)=3$ the discriminant curve is the union of three smooth branches $D_i(u,v)=0$, $1\leq i\leq 3$, with intersection number  $i_0(D_l,D_r)=\frac{2s_1+\lambda_i}{3}  $ for $l\neq r$.
\end{enumerate}
\end{enumerate}
\end{prop}
\noindent \begin{proof}
\noindent Suppose $\lambda_i=0$. By (\ref{eq:f}) the implicit equation of the normal form is  $f_1(x,y)=y^4-x^{s_1}$.  Then $(f_1)_y(x,y)=4y^3$ whose Newton-Puiseux root is $y = 0$, with multiplicity three. Hence, after (\ref{disc-e}) we have $D(u,v)=(v-f_1(u,0))^3=(v+u^{s_1})^3$, that is the discriminant curve is a triple smooth branch.

\smallskip

\noindent  Suppose now $\lambda_i\neq 0$. Then $\lambda_i=2{s_1}-4j$ for $2\leq j \leq [\frac{{s_1}}{4}]$  and the normal form of $C$ is 
$x=t^4,\;\;y= \sigma_i (t)$ for $2\leq i\leq 4$. 
\medskip

\noindent  From  the implicit equations $f_i(x,y)$, $2\leq i \leq 4$, given in (\ref{normal form 4}) and from the inequality (\ref{orden P}) we get that the Newton polygon of $ (f_i)_y(x,y)$ has only one compact edge whose vertices are $(0,3)$ and $(s_1-j,0)$. All the parametrisations of $ (f_i)_y(x,y)=0$, for $2 \leq i \leq 4$, have the same  order and we can write them by 
$\gamma_r(u)=\varepsilon u^{\frac{s_1-j}{3}}+\cdots,$ where $\varepsilon$ is a $3$th-root of the unity. From  $(\ref{disc-e})$ and for a fix $i\in \{2,3,4\}$, we have

\[D(u,v)=\prod_{(f_i)_y(\gamma_r)=0}(v-f_i(u,\gamma_r(u))),\]

and considering the development of $f_i(u,\gamma_r(u))$ we obtain:
\[
D(u,v)=\prod_{\varepsilon^3=1}(v+u^{s_1}+3\varepsilon u^{4\left(\frac{s_1-j}{3}\right)}+\cdots)=\prod_{\varepsilon^3=1}(v+u^{s_1}+3\varepsilon u^{\left(\frac{2s_1+\lambda_i}{3}\right)}+\cdots),
\]

so $\Zer(D)=\{\eta_i(\varepsilon):=-u^{s_1}-3\varepsilon u^{\left(\frac{2s_1+\lambda_i}{3}\right)}+\cdots\}_{\varepsilon^3=1}$ and 
$\ord(\eta_l-\eta_r)=\frac{2s_1+\lambda_i}{3}$, for $1\leq l\neq r\leq 3$. The topological type of  $D(u,v)=0$
 is determined by the semigroup of the branch $C\equiv f_i(x,y)=0$ and its  Zariski analytical invariant $\lambda_i = 2s_1-4j$  for $2\leq j \leq [\frac{{s_1}}{4}]$. We distinguish two cases:
if $3$ and $2s_1+\lambda_i$ are coprime then the discriminant $D(u,v)=0$ is a branch of semigroup $\langle 3,2s_1+\lambda_i\rangle$. Otherwise, by (\ref{lisas}), we conclude that the discriminant curve is the union of three different smooth branches $D_l(u,v)=0$ with intersection multiplicity $i_0(D_l,D_r)=\frac{2s_1+\lambda_i}{3}$ for $l\neq r$.

\medskip

\noindent In all cases the Newton polygon of $D(u,v)$ is elementary with vertices $(0,3)$ and $(3s_1,0)$. The polynomial associated with its compact edge is $(z+1)^3$, so the discriminant curve $D(u,v)=0$ is degenerate.
\end{proof}

\begin{prop}
\label{sigma5}
Let $C\equiv f(x,y)=0$ be  a branch belonging to the family NF 4.$5$. Then the discriminant curve $D(u,v)=0$ is degenerate and its topological type is determined by the semigroup $S(f)=\langle 4,s_1\rangle$, the Zariski invariant $\lambda_5$ and at most two other analytical invariants of $C$.

\end{prop}
\noindent \begin{proof}

The implicit equation of $C$ has the form $f_5(x,y)=y^4+P(x)y^2+Q(x)y+R(x)$, 
where $P(x), Q(x), R(x)$ are as in (\ref{P5}), (\ref{Q5}) and (\ref{R5}). 
Hence $(f_5)_y(x,y)=4y^3+2P(x)y+Q(x)$.

\medskip

 We distinguish different cases:\\
 
{\bf Case A.} \underline{ If $a_i=0$ in $\sigma_5(t)$ for all $i$ }: then  we have that $(f_5)_y(x,y)=4y^3-8x^{s_1-j}y=4y(y^2-2x^{s_1-j})$ whose Newton-Puiseux roots are $\left\{0, \pm \sqrt{2}x^{\frac{s_1-j}{2}}\right\}$. Therefore $f_5(u,0)=R(u)=-u^{s_1}+2u^{2(s_1-j)}-u^{3s_1-4j}$ and $f_5(u,\pm \sqrt{2}u^{\frac{s_1-j}{2}})=-u^{s_1}-2u^{2(s_1-j)}-u^{3s_1-4j}$. From (\ref{disc-e}) we conclude that the discriminant curve is the union of three smooth curves $D_l(u,v)=0,$ where $D_1(u,v)=v+u^{s_1}-2u^{2(s_1-j)}+u^{3s_1-4j}$, $D_2(u,v)=D_3(u,v)=v+u^{s_1}+2u^{2(s_1-j)}+u^{3s_1-4j}$, and $i_0(D_1,D_l)=2(s_1-j)$ for $l\in\{2,3\}$.\\

{\bf Case B.} \underline{ If $a_k\neq 0$  
\hbox{\rm and } $a_{k+l}=0$ in  $\sigma_5(t)\; \hbox{\rm for }l>0$}:
Then  
\begin{eqnarray*}(f_5)_y(x,y)&=&
4y^3+2(-4x^{s_1-j}-2a_{k}^2x^{s_1-2(j-[\frac{s_1}{4}]-k)}+\cdots)y\\
&-&4a_k\left(x^{s_1-j+[\frac{s_1}{4}]+k}+x^{2(s_1-j)-(j-[\frac{s_1}{4}]-k)}+\cdots\right).\\
\end{eqnarray*}

So the Newton polygon of $(f_5)_y(x,y)$ depends on the position of the point $M=(s_i-j,1)$ with respect to the line passing by $E=(0,3)$ and $F=(s_1-j+[\frac{s_1}{4}]+k,0)$. We get three situations:
 
\medskip
 	
{\bf B.1.} If $\frac{2}{s_1-j}<\frac{1}{[\frac{s_1}{4}]+k}$ then ${\mathcal N}((f_5)_y)$ has only one compact edge of vertices $E$ and $F$. Hence the order of the Newton-Puiseux roots $\{\gamma_i\}_{i=1}^3$ of $(f_5)_y(x,y)=0$ equals $\frac{s_1-j+[\frac{s_1}{4}]+k}{3}$. After (\ref{disc-e}) we obtain

\[D(u,v)=\prod_{w_i^3=1}(v-3a^{\frac{4}{3}}w_iu^{4\left(\frac{s_1-j+[\frac{s_1}{4}]+k}{3}\right)}+u^{s_1}+\cdots),\] 
for some nonzero complex number $a$ and where $w_{i}$ is a cubic root of the unity. If $\gcd(3,(s_1-j+[\frac{s_1}{4}]+k))=1$ then the discriminant curve is irreducible with semigroup $\langle 3,4(s_1-j+[\frac{s_1}{4}]+k) \rangle.$ On the other case,
 we get $\gcd(3,(s_1-j+[\frac{s_1}{4}]+k))=3$ and the discriminant curve is the union of three smooth curves $D_l(x,y)=0$ such that
 $i_0(D_l,D_r)=\frac{4(s_1-j+[\frac{s_1}{4}]+k)}{3}$.
 
 \smallskip
 
{\bf B.2.} If $\frac{2}{s_1-j}>\frac{1}{[\frac{s_1}{4}]+k}$ then ${\mathcal N}((f_5)_y)$ has two compact edges of vertices $M$, $E$ and $F$. By (\ref{disc-e})  we get
 	
\[D(u,v)=(v+u^{s_1} -2u^{2(s_1-j)} +\cdots    ) (v+u^{s_1} +2u^{2(s_1-j)} \pm 4\sqrt{2}a_ku^{\frac{3}{2}(s_1-j)+[\frac{s_1}{4}]+k}+\cdots).
\]

 If  $s_1-j$ is odd then the discriminant curve is the union of a smooth  branch  $D_1(u,v)=0$ and a singular branch $D_2(x,y)=0$ with semigroup $\langle 2,3(s_1-j)+2([\frac{s_1}{4}]+k) \rangle$. Moreover
$i_0(D_1,D_2)=4(s_1-j).$ Otherwise, if $s_1-j$ is even then the discriminant curve is the union of three smooth branches $D_l(x,y)=0$ such that $i_0(D_1,D_l)=2(s_1-j)$ for $2\leq l \leq 3$ and $i_0(D_2,D_3)=\frac{3}{2}(s_1-j)+[\frac{s_1}{4}]+k$ .\\
 
 \medskip

{\bf B.3.} If $\frac{2}{s_1-j}=\frac{1}{[\frac{s_1}{4}]+k}$ then ${\mathcal N}((f_5)_y)$ has only one compact edge of vertices $E$ and $F$ and $M$ is an interior point of this edge. The polynomial in one variable, associated with the compact edge of  ${\mathcal N}((f_5)_y)$, is  $p(z)=z^3-2z-a_k$. After the $z$-discriminant of $p(z)$ we have that the roots of $p(z)$ are  simple if and only if $a_k\neq \pm \left(\frac{4\sqrt{6}}{9}\right)$. We will study both cases:
 	
\hspace{0.5cm} {\bf B.3.1}  Suppose that $a_k\neq \pm \left(\frac{4\sqrt{6}}{9}\right)$. Denote by $z_i$ the three different roots of the polynomial $p(z)$.
For any  $\gamma_i(x)\in \Zer((f_5)_y)$ we have:\\

$f(u,\gamma_i(u))=-u^{s_1}+q(z_i)u^{2(s_1-j)}+\cdots$, where $q(z)=z^4-4z^2-4a_kz+2$.\\

 Since $z^3-2z=a_k$ then $q(z)=-3z^4+4z^2+2$ and if $z_r\neq z_l$ then $q(z_r)\neq q(z_l)$. Hence
$D(u,v)=\prod_{{i}}(z+u^s_1+q(z_i)u^{2(s_1-j)}+\cdots)$ and the discriminant curve $D(u,v)=0$ is the union of three smooth branches $D_l(x,y)=0$ such that $i_0(D_l,D_r)=2(s_1-j)$.
 	
\hspace{0.5cm} {\bf B.3.2} Suppose that $a_k= \frac{4\sqrt{6}}{9}$. The polynomial $p(z)$ has $z_1=2\sqrt{\frac{2}{3}}$ as a simple root and $z_2=-\sqrt{\frac{2}{3}}$ as a double root.  If $\gamma_i\in \Zer((f_5)_y)$ corresponds to $z_i$ then 
\[
\begin{array}{l}
 f_5(u,\gamma_1(u))= u^{s_1}+q(z_1)u^{2(s_1-j)}+\cdots,\\
 f_5(u,\gamma_{2}(u))= u^{s_1}+q(z_2)u^{2(s_1-j)}+r_{2}u^{2(s_1-j)+\frac{s_1-2j}{2}}\cdots,\\
 f_5(u,\gamma_{3}(u))= u^{s_1}+q(z_2)u^{2(s_1-j)}+r_{3}u^{2(s_1-j)+\frac{s_1-2j}{2}}\cdots,\\
\end{array}
\]
where $r_2,r_3 \in \C$ are different. Observe that $q(z_1)\neq q(z_2)$.  Hence, if $s_1-2j$ is odd then the discriminant curve is the union of a smooth branch  $D_1(u,v)=0$ and a singular branch $D_2(u,v)=0$ with semigroup $\langle 2, 5s_1-6j \rangle,$ where  $i_0(D_1,D_2)=4(s_1-j)$. Otherwise, if $s_1-2j$ is even then the discriminant curve is the union of three smooth branches $D_l(u,v)=0$ such that
 	$i_0(D_1,D_l)=2(s_1-j)$, for $2\leq l \leq 3$ and $i_{0}(D_2,D_3)=\frac{5s_1-7j}{2}$.

 \hspace{0.5cm} {\bf B.3.3} Suppose that $a_k= -\frac{4\sqrt{6}}{9}$. The polynomial $p(z)$ has $2\sqrt{\frac{2}{3}}$ as a double root and $-\sqrt{\frac{2}{3}}$ as a simple root. After a similar procedure we conclude, in this case, that the topological type of the discriminant curve, is as in B.3.2.\\
 
 {\bf Case C.} \underline{If $a_k\neq 0\neq a_{k+s}$:} Then we have  $(f_5)_y(x,y)=4y^3+2P(x)y+Q(x)$.

The Newton polygon of $(f_5)_y(x,y)$ depends on the position of the point $M=(s_i-j,1)$ with respect to the line passing by $E=(0,3)$ and $F=(s_1-j+[\frac{s_1}{4}]+k,0)$. We have the following situations:\\

{\bf C.1.}  If $\frac{2}{s_1-j}<\frac{1}{[\frac{s_1}{4}]+k}$ then the topological type of the discriminant curve $D(u,v)=0$ is as in the case {\bf B.1}.

\medskip

{\bf C.2.} If $\frac{2}{s_1-j}>\frac{1}{[\frac{s_1}{4}]+k}$ then the topological type of de discriminant curve is as in the case {\bf B.2}.

\medskip

{\bf C.3.} If $\frac{2}{s_1-j}=\frac{1}{[\frac{s_1}{4}]+k}$ then the Newton polygon of $(f_5)_y(x,y)$ has only one compact edge containing the points $E,F,M$,  as in the case {\bf B.3}. The polynomial associated with  this compact edge is $p(z)=z^3-2z-a_k$ whose roots are simple if and only if 
$a_k\neq \pm \left(\frac{4\sqrt{6}}{9}\right)$. Let us study the different cases:

 \hspace{0.5cm}	{\bf C.3.1.} If $a_k\neq \pm \left(\frac{4\sqrt{6}}{9}\right)$ then the topological type of the discriminant curve $D(u,v)=0$ is as in {\bf B.3.1.}
 
\hspace{0.5cm}	{\bf C.3.2.} Suppose that $a_k=\pm \left(\frac{4\sqrt{6}}{9}\right)$.
The polynomial $p(z)$ has $z_1=2\sqrt{\frac{2}{3}}$ as a simple root and $z_2=z_3=-\sqrt{\frac{2}{3}}$ as a double root.  If $\gamma_i\in \Zer((f_5)_y)$ corresponds to $z_i$ then we can write

\begin{equation}
\label{param}
\gamma_i=z_iu^{[\frac{s_1}{4}]+k}+\cdots 
\end{equation}
Hence
\[
\begin{array}{l}
\delta_1:= f(u,\gamma_1(u))=- u^{s_1}+q(z_1)u^{2(s_1-j)}+\cdots,\\
\delta_2:= f(u,\gamma_{2}(u))= -u^{s_1}+q(z_2)u^{2(s_1-j)}+\cdots,\\
\delta_3:= f(u,\gamma_{3}(u))= -u^{s_1}+q(z_2)u^{2(s_1-j)}+\cdots,\\
\end{array}
\]
where $\delta_2,\delta_3\in \Zer D(u,v)$ are different. Observe that $q(z_1)\neq q(z_2)$. So $\ord\left(\delta_1-\delta_l \right)=2(s_1-j)$ for $2\leq l\leq 3$. Let us determine $\ord\left(\delta_2-\delta_3\right)$.  For that we need to precise new terms in $\delta_2$ and $\delta_3$. We apply the Newton procedure for $(f_5)_y$: let $y_1$ a new variable. Substituting $(x,y):=(x,x^{[\frac{s_1}{4}]+k}(z_2+y_1))$ in $(f_5)_y(x,y)$ we get
 
\[
(f_5)_y(x,x^{[\frac{s_1}{4}]+k}(z_2+y_1))=x^{3([\frac{s_1}{4}]+k)}g(x,y_1),
\]

with	 	
\begin{eqnarray}
g(x,y_1)&=&4\left[\left(\frac{2}{3}\right)z_2+2y_1+3z_2y_1^2+y_1^3\right]-(4a_k+4a_{k+s}x^s+4a_kx^{s_1-2j}+\cdots) \nonumber\\	
&-&[8+4a_k^2x^{s_1-2j}+8a_ka_{k+s}x^{s_1-2j+s}+4a_{k+s}^2x^{s_1-2j+2s}+\cdots  ]y_1
\nonumber\\\  
  &-&[8+4a_k^2x^{s_1-2j}+8a_ka_{k+s}x^{s_1-2j+s}+
 	 4a_{k+s}^2x^{s_1-2j+2s}+\cdots  ]z_2 \nonumber\\
&=&4\left[3z_2y_1^2+y_1^3\right]-(4a_{k+s}x^s+4a_kx^{s_1-2j}+\cdots) \nonumber\\	
&-&[4a_k^2x^{s_1-2j}+8a_ka_{k+s}x^{s_1-2j+s}+4a_{k+s}^2x^{s_1-2j+2s}+\cdots  ]y_1\nonumber \\ 
  &-&[4a_k^2x^{s_1-2j}+8a_ka_{k+s}x^{s_1-2j+s}+
 	 4a_{k+s}^2x^{s_1-2j+2s}+\cdots  ]z_2 \nonumber,\label{derivada}\\	 
\end{eqnarray}
where the last equality follows from $p(z_2)=0$.\\

Hence, in the next step of the Newton procedure it is enough to consider the polynomial

\begin{equation}
\label{compara}
 	 H(x,y_1)=12z_2 y_1^2-4a_k^2x^{s_1-2j}y_1+(-4z_2a_k^2- 4a_k)x^{s_1-2j}-4a_{k+s}x^s.
\end{equation}
 
The topological type of the discriminant curve will depend on the relation between $s$ and $s_1-2j$:\\
 	
 \hspace{1.2cm}{\bf C.3.2.1} Suppose that $s_1-2j>s$. We get:\\
 
 \[\gamma_2(u)=-\frac{\sqrt{6}}{3}u^{\frac{s_1-j}{2}}+\frac{\sqrt{-a_{k+s}}}{\sqrt[4]{6}}u^{\frac{s_1-j}{2}+\frac{s}{2}}+\cdots,
 \]
\[
\gamma_3(u)=-\frac{\sqrt{6}}{3}u^{\frac{s_1-j}{2}}-\frac{\sqrt{-a_{k+s}}}{\sqrt[4]{6}}u^{\frac{s_1-j}{2}+\frac{s}{2}}+\cdots
\]	
	
Hence

\[
f_5(u,\gamma_2(u))=-u^s_1+q(z_2)u^{2(s_1-j)}+
 p(z_2)u^{ 2(s_1-j)+\frac{s}{2}}+l(a_{k+s})u^{3s_1-4j}+\frac{8}{3}ca_{k+s}u^{2(s_1-j)+\frac{3}{2}s } +\cdots,
 \]
and
\[
f_5(u,\gamma_3(u))=-u^s_1+q(z_2)u^{2(s_1-j)}-
 p(z_2)u^{ 2(s_1-j)+\frac{s}{2}}+l(a_{k+s})u^{3s_1-4j}-\frac{8}{3}ca_{k+s}u^{2(s_1-j)+\frac{3}{2}s } +\cdots,
 \]

 where $c:=\frac{\sqrt{-a_{k+s}}}{\sqrt[4]{6}}$, y $l(z)=-4z_2z-1$. As a consequence,
 when $s$ is odd then the discriminant curve $D(u,v)=0 $ is the union of a smooth branch $D_1(u,v)=0$ and a singular branch with semigroup $\langle 2,4(s_1-j)+3s\rangle$, with $i_0(D_1,D_2)=4(s_1-j)$ . Otherwise, if $s$ is even then the discriminant curve $D(u,v)=0 $ is the union of three smooth braches $D_l(u,v)=0$  such that $i_0(D_1,D_l)=2(s_1-j)$ for $2\leq l\leq 3$ and $i_0(D_2,D_3)=2(s_1-j)+\frac{3}{2}s$.

\medskip

 \hspace{1.2cm}{\bf C.3.2.2} Suppose that $s_1-2j<s$. After
 (\ref{compara}) and for the next step of the Newton procedure we only need the polynomial 
 
\[
\bar H(Z)=12z_2Z^2-4(z_2a_{k}^2+a_k),
\]
whose roots are $\pm \frac{\sqrt{8}}{9}$. Let $d=\frac{\sqrt{8}}{9}$. 
We get:\\
 
 \[\gamma_2(u)=z_2u^{\left[\frac{s_1}{4}\right]+k}+du^{[\frac{s_1}{4}]+k+\frac{s_1-2j}{2}}+\cdots,
 \]
\[
\gamma_3(u)=z_2u^{\left[\frac{s_1}{4}\right]+k}-du^{[\frac{s_1}{4}]+k+\frac{s_1-2j}{2}}+\cdots.\]	
	
Hence

\[
f_5(u,\gamma_2(u))=-u^s_1+q(z_2)u^{2(s_1-j)}+
 p(z_2)u^{ 2(s_1-j)+\frac{s}{2}}+l(a_{k})u^{3s_1-4j}+\frac{80}{81}z_2du^{\frac{7s_1-10j}{2}}+\cdots,
 \]
and
\[
f_5(u,\gamma_3(u))=-u^s_1+q(z_2)u^{2(s_1-j)}-
 p(z_2)u^{ 2(s_1-j)+\frac{s}{2}}+l(a_{k})u^{3s_1-4j}-\frac{80}{81}z_2du^{\frac{7s_1-10j}{2}}+\cdots
 \]

Since $4$ and $s_1$ are coprime then  $7s_1-10$ is odd and the discriminant curve $D(u,v)=0 $ is the union of a smooth branch $D_1(u,v)=0$ and a singular branch with semigroup $\langle 2,7s_1-10j\rangle$, where $i_0(D_1,D_2)=4(s_1-j)$.

\medskip

 \hspace{1.2cm}{\bf C.3.2.3} Suppose that $s_1-2j=s$.   After (\ref{compara}), in order to obtain the next term in the power series $\gamma_i$ it is enough to  consider the polynomial 
\[\bar H(Z)=12z_2Z^2-4(z_2a_{k}^2+a_{k+s}+a_k).\]
The topological type of the discriminant will depend on the value of   $a_{k+s}$:

 \hspace{1.25cm}{\bf C.3.2.3.1} For $a_{k+s}\neq -4\frac{\sqrt{6}}{81}$, we get 
\[\bar H(Z)=12z_2Z^2-4\left(a_{k+s}+4\frac{\sqrt{6}}{81}\right),\]
 whose roots are $\pm b:=\pm \left (\frac{\sqrt{-d_{k+s}}}{\sqrt[4]{6}}\right)$,
 where	 $d_{k+s}=a_{k+s}+4\frac{\sqrt{6}}{81}$. Hence

 \[\gamma_2(u)=z_2u^{\left[\frac{s_1}{4}\right]+k}+bu^{[\frac{s_1}{4}]+k+\frac{s}{2}}+\cdots,
 \]
\[
\gamma_3(u)=z_2u^{\left[\frac{s_1}{4}\right]+k}-bu^{[\frac{s_1}{4}]+k+\frac{s}{2}}+\cdots 
\]	
	
We conclude that

\[
f_5(u,\gamma_2(u))=-u^s_1+q(z_2)u^{2(s_1-j)}+
l(a_{k+s})u^{3s_1-4j}+\frac{80}{81}z_2u^{2s_1-j+\frac{3}{2}s}+\cdots,
 \]
and
\[
f_5(u,\gamma_3(u))=-u^s_1+q(z_2)u^{2(s_1-j)}+l(a_{k+s})u^{3s_1-4j}-\frac{80}{81}z_2u^{2s_1-j+\frac{3}{2}s}+\cdots 
 \]

So, if $s$ is odd then   $D(u,v)=0$ is the union of a smooth branch $D_1(u,v)=0$ and  a singular branch of semigroup $\langle 2,7s_1-10j \rangle $, with $i_0(D_1,D_2)=4(s_1-j)$.
 On the other case, if $s$ is even then   $D(u,v)=0$ is the union of three smooth branches $D_l(u,v)=0$ such that $i_0(D_1,D_l)=2(s_1-j)$, and $i_0(D_2,D_3)= \frac{7s_1-10j}{2}$.

\medskip

 \hspace{1.25cm}{\bf C.3.2.3.2} For $a_{k+s}= -4\frac{\sqrt{6}}{81}$, after (\ref{derivada}),
in order to obtain the next term in the power series $\gamma_i$ it is enough to  consider the polynomial  

 \[
\bar H(Z)=12z_2Z^2-4a_k^2Z-4a_{k+s}(1+2a_k z_2 ),
 \]
 whose roots $t_1,t_2$ are simple. Hence 

 \[\gamma_2(u)=z_2u^{\left[\frac{s_1}{4}+k\right]+k}+t_1u^{[\frac{s_1}{4}+k]+k+s}+\cdots,
 \]
 and
\[
\gamma_3(u)=z_2u^{\left[\frac{s_1}{4}+k\right]+k}-t_2u^{[\frac{s_1}{4}+k]+k+s}+\cdots.
\]	
	
We conclude that

\[
f_5(u,\gamma_2(u))=-u^s_1+q(z_2)u^{2(s_1-j)}-u^{2(s_1-j)+s}+
h(t_1)u^{2(s_1-j)+s2}+\cdots,
 \]
and
\[
f_5(u,\gamma_3(u))=-u^s_1+q(z_2)u^{2(s_1-j)}-u^{2(s_1-j)+s}+
h(t_2)u^{2(s_1-j)+s2}+\cdots,
 \]
where $h(z)=-4a_{k+s}z+a_k$.	The discriminant curve  $D(u,v)=0$ is the union of three smooth branches $D_l(u,v)=0$ such that $i_0(D_1,D_l)=2(s_1-j)$ and $i_0(D_2,D_3)=4s_1-6j$.  	
\end{proof}

\medskip

From the above computations we achieve:
\begin{teor}
Let $f(x,y)=0$ be a plane branch of multiplicity $n$. We have
\begin{enumerate}
\item If $n=2$ then the discriminant curve $D(u,v)=0$ is non degenerate.
\item If $n=3$ then the discriminant curve $D(u,v)=0$ is degenerate and its topological type depends on the semigroup $S(f)$ and its Zariski invariant (when it is not zero).
\item If $n=4$ then the discriminant curve $D(u,v)=0$ is degenerate and its topological type depends on the semigroup $S(f)$, the  Zariski invariant (when it is not zero) and at most two other analytical invariants.
\end{enumerate}
\end{teor}

\begin{obs} $\;$
Since there is an explicit normal form for branches of  multiplicity 2, 3 or 4, the description of the topological type of the discriminant curve of such a branch  was possible. The normal form for branches of multiplicity greater than 4 is not completely determined. 

\end{obs}

\section{Discriminant of branches  $C$ with $\mu(C)- \tau(C) \leq 2$}

Let $C:f(x,y)=0$ be a plane branch. Put $r(C):=\mu(C)-\tau(C)$, where $\mu(C)$ and $\tau(C)$ are the Milnor number and the Tjurina number of $C$, respectively. Observe that $r(C)$ is a nonnegative integer number.
In zero characteristic, from  \cite[Theorem 4]{Zariski-1966} we get $r(C)=0$ if and only if $C$ is analytically equivalent to the curve $y^{s_0}-x^{s_1}=0$, for two coprime integers $s_0$ and $s_1$ greater than one. Later, in \cite{B-Hefez}, the authors describe all plane branches defined over an algebraically closed field of characteristic zero, modulo analytic equivalence, having the property that the difference between their Milnor and Tjurina numbers is one or two. In particular the authors determined the normal forms of the branches of  this family, which show us that the Zariski invariant of these branches are determined by the two first generators of their semigroup. By  \cite[Corollary 5]{B-Hefez} we know that if $r(C)\neq 0$ then $r(C)\geq 2^{g-1}$. Hence, if $r(C)=1$ then $g=1$ and if $r(C)=2$ then $g\leq 2$.   In this section we will describe the topological type of the discriminant curve $D(u,v)=0$ of branches $C$ with $\mu(C)- \tau(C) \leq 2$. 

\begin{teor}
\label{TM}
Let $C:f(x,y)=0$ be a plane branch with $r(C):=\mu(C)-\tau(C)\leq 2$. Then the discriminant curve $D(u,v)=0$ is degenerate and its topological type is given by the semigroup $S(f)$.
\end{teor}
\begin{proof}
Suppose first that $r(C)=1$. By \cite[Corollary 8]{B-Hefez} the branch $C$ is analytically equivalent to the curve  defined by the equation $f(x,y)=y^{s_0}-x^{s_1}+x^{s_1-2}y^{s_0-2}$, where $2\leq s_0<s_1$ are coprime integers. Hence  $f_y(x,y)=s_0y^{s_0-1}+(s_0-2)x^{s_1-2}y^{s_0-3}=y^{s_0-3}(s_0y^2+(s_0-2)x^{s_1-2})$, which Newton-Puiseux roots are $\alpha_1=0$ (with multiplicity $s_0-3$ ), $\alpha_2=\sqrt{\frac{2-s_0}{s_0}}x^{\frac{s_1-2}{2}}$ and 
$\alpha_3=-\sqrt{\frac{2-s_0}{s_0}}x^{\frac{s_1-2}{2}}$. After (\ref{disc-e}) we obtain that the Newton-Puiseux roots of the discriminant curve are

\begin{enumerate}
\item $\delta_1=u^{s_1}$ with multiplicity ${s_0-3}$,
\item $\delta_2=-u^{s_1}+\left(\left(\sqrt{\frac{2-s_0}{s_0}}\right)^{s_0}+\left(\sqrt{\frac{2-s_0}{s_0}}\right)^{s_0-2}\right)u^{\frac{(s_1-2)s_0}{2}}$,
\item $\delta_3=-u^{s_1}-\left(\left(\sqrt{\frac{2-s_0}{s_0}}\right)^{s_0}+\left(\sqrt{\frac{2-s_0}{s_0}}\right)^{s_0-2}\right)u^{\frac{(s_1-2)s_0}{2}}$.
\end{enumerate}

Hence, if $s_1$ and $s_0$ are odd  then  the discriminant curve is given by $D(u,v)=D_1(u,v)^{s_0-3}D_2(u,v),$ 
where $D_1(u,v)=(v-u^{s_1})$ and $D_2(u,v)$ is a branch of semigroup $\langle 2, (s_1-2)s_0\rangle$ and the intersection multiplicity $i_0(D_1,D_2)=\min\left\{s_1,(s_1-2)n\right\}$. Otherwise $D(u,v) $ is the product of $D_1(u,v)^{s_0-3}$ and two smooth branches $D_2(u,v)$ and $D_3(u,v)$, where $i_0(D_1,D_k)=\min\left\{s_1,\frac{(s_1-2)s_0}{2}\right\}$ for $2\leq k \leq 3$ and $i_0(D_2,D_3)=\frac{(s_1-2)s_0}{2}$.

\medskip
Suppose now that $r(C)=2$. In this case we get $g\leq 2$. If $g=2$ then by 
\cite[Corollary 13]{B-Hefez} the multiplicity of $C$ is 4 and we study this case in Proposition \ref{4NG}.

\medskip

For $g=1$, after  \cite[Theorem 17, Corollary 18]{B-Hefez}, we have two normal forms
\begin{enumerate}
\item [(A)] $f(x,y)=y^{s_0}-x^{s_1}+x^{s_1-3}y^{s_0-2},$ with $2<s_0<s_1$.
\item [(B)] $f(x,y)=y^{s_0}-x^{s_1}+x^{s_1-2}y^{s_0-3}+\left(\sum_{k \geq 2}^{2+[\frac{s_1}{s_0}]}a_{k}x^{s_1-k}\right)y^{s_0-2},$ with $4\leq s_0<s_1$, $\frac{2s_0}{s_0-3}<s_1$ and $a_k\in \C$.
\end{enumerate}

\medskip
In the case (A) we get $f_y(x,y)=y^{s_0-3}(s_0y^2+(s_0-2)x^{s_1-3}),$ which Newton-Puiseux roots are $\alpha_1=0$ (with multiplicity $s_0-3$ ), $\alpha_2=\sqrt{\frac{2-s_0}{s_0}}x^{\frac{s_1-3}{2}}$ and 
$\alpha_3=-\sqrt{\frac{2-s_0}{s_0}}x^{\frac{s_1-3}{2}}$. After (\ref{disc-e}) we obtain that the Newton-Puiseux roots of the discriminant curve are

\begin{enumerate}
\item $\delta_1=u^{s_1}$ with multiplicity ${s_0-3}$,
\item $\delta_2=-u^{s_1}+\left(\left(\sqrt{\frac{2-s_0}{s_0}}\right)^{s_0}+\left(\sqrt{\frac{2-s_0}{s_0}}\right)^{s_0-2}\right)u^{\frac{(s_1-3)s_0}{2}}$,
\item $\delta_3=-u^{s_1}-\left(\left(\sqrt{\frac{2-s_0}{s_0}}\right)^{s_0}+\left(\sqrt{\frac{2-s_0}{s_0}}\right)^{s_0-2}\right)u^{\frac{(s_1-3)s_0}{2}}$.
\end{enumerate}

Hence, if  $s_0$ is odd and $s_1$ is even then the discriminant curve is given by $D(u,v)=D_1(u,v)^{s_0-3}D_2(u,v),$ 
where $D_1(u,v)=(v-u^{s_1})$ and $D_2(u,v)$ is a branch of semigroup $\langle 2, (s_1-3)s_0\rangle$ and the intersection multiplicity $i_0(D_1,D_2)=\min\left\{s_1,(s_1-3)s_0\right\}$. Otherwise $D(u,v) $ is the product of $D_1(u,v)^{s_0-3}$ and two smooth branches $D_2(u,v)$ and $D_3(u,v)$, where $i_0(D_1,D_k)=\min\left\{s_1,\frac{(s_1-3)s_0}{2}\right\}$ for $2\leq k \leq 3$ and $i_0(D_2,D_3)=\frac{(s_1-3)s_0}{2}$.
\medskip

In the case (B) we have $f_y(x,y)=y^{s_0-4}\left(s_0y^3+(s_0-3)x^{s_1-2}+(s_0-2)\left(\sum_{k \geq 2}^{2+[\frac{s_1}{s_0}]}a_{k}x^{s_1-k}\right)y\right),$ hence its Newton polygon coincides with the Newton polygon determined by   $(0,3)$, 
$(s_1-2,0)$ and $(s_1-2-[\frac{s_1}{s_0}],1)$. But, after the inequality $\frac{2s_0}{s_0-3}<s_1$, we get that this Newton polygon has only two points which are its vertices: $(0,3)$ and $(s_1-2,0)$. The Newton-Puiseux roots of $f_y(x,y)=0$ are $\alpha_1=0$ (with multiplicity $s_0-4$) and $\alpha_i=\xi_i\sqrt[3]{\frac{s_0-3}{s_0}}u^{\frac{s_0-2}{3}}+\cdots$, where $\xi_i$ is a cubic root of the unity, $1\leq i\leq 3$. Then, after (\ref{disc-e}) the Newton-Puiseux roots of the discriminant curve are $\delta_1=u^{s_1}$ with multiplicity ${s_0-4}$ and

\begin{eqnarray*}\delta_i&=&-u^{s_1}+\left[\left(\xi_i\sqrt[3]{\frac{s_0-3}{s_0}}\right)^{s_0}+\left(\xi_i\sqrt[3]{\frac{s_0-3}{s_0}}\right)^{s_0-3}\right]u^{\frac{(s_1-2)s_0}{3}}\\
&+&\left(\xi_i\sqrt[3]{\frac{s_0-3}{s_0}}\right)^{s_0-2}a_{2+[\frac{s_1}{s_0}]}u^{\frac{(s_1-2)(s_0-2)}{3}+(s_1-2)-[\frac{s_1}{s_0}]}+\cdots,\\
\end{eqnarray*} for  $1\leq i\leq 3$.

If $s_1-2$ (respectively $s_0$) and $3$ are coprime then the discriminant curve is given by $D(u,v)=D_1(u,v)^{s_0-4}D_2(u,v),$ 
where $D_1(u,v)=(v-u^{s_1})$ and $D_2(u,v)$ is a branch of semigroup $\langle 3, (s_1-2)s_0\rangle$ and, after Halphen-Zeuthen formula, the intersection multiplicity $i_0(D_1,D_2)=3s_1$. Otherwise $D(u,v) $ is the product of $D_1(u,v)^{s_0-4}$ and three smooth branches $D_k(u,v)$, $2\leq k\leq 4$, where $i_0(D_1,D_k)=s_1$ for $2\leq k \leq 4$ and $i_0(D_l,D_r)=\frac{(s_1-2)s_0}{3}$, for $2\leq l\neq r\leq 4$.
\end{proof}

\begin{remark}
Observe that in Theorem \ref{TM} the Case (B) with $s_0=4$ coincides with the case $\sigma_3$ in  (\ref{sigmas}) for $j=2$. Hence this case was studied in Proposition \ref{1-4}.
\end{remark}

\section{Tables}
\label{tablas}
\noindent The following tables collect the topological type of the discriminants for branches studied in this paper:

{\footnotesize
\begin{table}[H]
\begin{center}
\begin{tabular}{|l|l|}
\hline \multicolumn{2}{|c|}{{\bf Multiplicity $2$}} \\
	\hline {\bf Normal form} &{\bf Discriminant $D(u,v)$} \\
	\hline
	\hline $\lambda=0$& $v+u^{s_1}$ \\
	\hline 
\end{tabular}
\caption{Discriminants of branches of semigroup $\langle 2,s_1\rangle$}
\label{tm2}	
\end{center}
\end{table}
}

{\footnotesize
\begin{table}[H]
\begin{center}
\begin{tabular}{|l|l|}
\hline \multicolumn{2}{|c|}{{\bf Multiplicity $3$}} \\
	\hline {\bf Normal form} &{\bf Discriminant $D(u,v)$} \\
	\hline
	\hline $\lambda=0$& $(v+u^{s_1})^2$ \\
	\hline $\lambda\neq0$&$D_1 D_2,\,\,\,\,\,\,m(D_i)=1,$   $\,\,\,\,$ $i_0(D_1,D_2)=\frac{s_1+\lambda}{2}$ $\;$ if $\gcd(2,s_1+\lambda)=2$ \\	
	&$D_1$, $\;\;\;\,\;\;\;\;S(D_1)=\langle 2,s_1+\lambda\rangle$    $\;\;\;\;\;\;\;\;\;\;\;\;\;\;\;\;\;\;\;\;\;\;$ if $\gcd(2,s_1+\lambda)=1$\\ 	\hline
\end{tabular}
\caption{Discriminants of branches of semigroup $\langle 3,s_1\rangle$}
\label{m3}	
\end{center}
\end{table}

\begin{table}[H]
\begin{center}
\begin{tabular}{|l|l|}
\hline \multicolumn{2}{|c|}{{\bf Multiplicity $4$ and $g=2$}} \\
	\hline {\bf Normal form} &{\bf Discriminant $D(u,v)$} \\
	\hline
	\hline $\;\;\lambda=s_2-6$& $D_1D_2,\,\,\,\,\,\,m(D_1)=1,\,\,\,\,\,\,S(D_2)=\langle 2,s_2\rangle,\,\,\,\,\,\,  i_0(D_1,D_2)=2s_1$ \\
	\hline
\end{tabular}
\caption{Discriminants of branches of semigroup $\langle 4,6,s_2\rangle$}
\label{m4-g=2}	
\end{center}
\end{table}
}

{\footnotesize
\begin{table}[H]
\begin{center}
\begin{tabular}{|l|l|}
\hline \multicolumn{2}{|c|}{{\bf Multiplicity $4$ and $g=1$}} \\
	\hline {\bf Normal form} &{\bf Discriminant $D(u,v)$} \\
	\hline
	\hline $\;\; \lambda=0$& $(v+u^{s_1})^3$ \\
	\hline $\begin{array}{l}\lambda=2s_1-4j\neq0\\ 2\leq j \leq [\frac{{s_1}}{4}]\end{array}$&$D_1 D_2 D_3,\,\,\,\,\,\,m(D_i)=1,$   $\,\,\,\,$ $i_0(D_l,D_r)=\frac{2s_1+\lambda_i}{3}$ $\;\;\;$ if $\gcd(3,2s_1+\lambda_i)=3$, \\	
	&$D_1$, $\,\;\;\;\;\;\;\;\;\;\;$ $S(D_1)=\langle 3,2s_1+\lambda_i\rangle $   $\;\;\;\;\;\;\;\;\;\;\;\;\;\;\;\;\;\;\;\;\;\;\;$ if $\gcd(3,2s_1+\lambda_i)=1$\\ 	
	\hline
	\hline $\begin{array}{l}\lambda=3s_1-4j\neq0\\ 2\leq j \leq [\frac{{s_1}}{2}]\end{array}$& See Table 
	\ref{m4-g=1 V}\\ \hline
\end{tabular}
\caption{Discriminants of branches of semigroup $\langle 4,s_1\rangle$}
\label{m4-g=1 I-IV}	
\end{center}
\end{table}
}

{\footnotesize
\begin{table}[H]
\begin{center}
\begin{tabular}{|l|l|}
\hline \multicolumn{2}{|c|}{{\bf Multiplicity $4$, $g=1$, $\lambda=3s_1-4j\neq0\;\;2\leq j \leq [\frac{{s_1}}{2}]$}} \\
\hline {\bf Depending on the analytical invariants} &{\bf Discriminant $D(u,v)$} \\
	\hline
	\hline $\;\;\lambda$ & $D_1D_2^2,\;\;\;m(D_i)=1,\;\;\;i_0(D_1,D_2)=2(s_1-j)$\\ \hline
	\hline $\begin{array}{l}\lambda\;\;\hbox{\rm and} \;\;2s_1-4(j-[\frac{s_1}{4}]-k)\\ \hbox{\rm where }k=\min\{j\,:\,a_j\neq 0\}\\ \end{array}$& see Table \ref{Case B}	(case B of Proposition \ref{sigma5})\\
	\hline $\begin{array}{l}\lambda,\;\;2s_1-4(j-[\frac{s_1}{4}]-k) \;\;\hbox{\rm and } 2s_1-4(j-[\frac{s_1}{4}]-k-s)\\
	\hbox{\rm where } k=\min\{j\,:\,a_j\neq 0\}, s=\min\{j\;:\;a_{k+j}\neq 0\,, j>0\} \end{array}$& see Table \ref{Case C} (case C of Proposition \ref{sigma5})	\\ \hline
\end{tabular}
\caption{}
\label{m4-g=1 V}	
\end{center}
\end{table}
}

{\footnotesize
\begin{table}[H]
\begin{center}
\begin{tabular}{|l|l|}
\hline \multicolumn{2}{|c|}{{\bf Case B of Proposition \ref{sigma5}}} \\
\hline $\;$&{\bf Discriminant $D(u,v)$} \\
	\hline
	\hline $\frac{2}{s_1-j}<\frac{1}{[\frac{s_1}{4}]+k}$& $\begin{array}{ll}D_1D_2D_3,\;\;\;m(D_l)=1,\;\;\;i_0(D_l,D_s)=\frac{4}{3}(s_1-j+[\frac{s_1}{4}]+k)& \hbox{\rm if} \gcd\{3,s_1-j+[\frac{s_1}{4}]+k\}=3\\
	D_1,\;\;\;\;\;\;\;\;\;\;\;\;S(D_1)=\langle 3, 4(s_1-j+[\frac{s_1}{4}]+k)\rangle&\hbox{\rm if} \gcd\{3,s_1-j+[\frac{s_1}{4}]+k\}=1\\
	\end{array}$\\ \hline
	\hline $\frac{2}{s_1-j}>\frac{1}{[\frac{s_1}{4}]+k}$& $\begin{array}{ll}D_1D_2D_3,\;\;\;m(D_l)=1,\;\;\;i_0(D_1,D_s)=2(s_1-j),\;\;i_0(D_2,D_3)=\frac{3}{2}(s_1-j)+[\frac{s_1}{4}]+k& \hbox{\rm if} \gcd\{2,s_1-j\}=2\\
	D_1D_2,\;\;\;\;\;\;\;m(D_1)=1\;\;\;\;S(D_2)=\langle2,3(s_1-j)+2([\frac{s_1}{4}]+k)\rangle\;\;\;i_0(D_1,D_2)=4(s_1-j)&\hbox{\rm if}
	 \gcd\{2,s_1-j\}=1\\
	\end{array}$\\
	\hline $\frac{2}{s_1-j}=\frac{1}{[\frac{s_1}{4}]+k}$& $\begin{array}{ll}D_1D_2D_3,\;\;\;m(D_l)=1,\;\;\;i_0(D_l,D_s)=2(s_1-j)& \hbox{\rm if } a_k\neq \pm \frac{4\sqrt{6}}{9}\\
	D_1D_2,\;\;\;\;\;\;\;m(D_1)=1,\;\;\;S(D_2)=\langle2,5s_1-6j\rangle, \;\;\;i_0(D_1,D_2)=4(s_1-j)&\hbox{\rm if } a_k=\pm \frac{4\sqrt{6}}{9}
	\\ 
	\end{array}$\\\hline
\end{tabular}
\caption{}
\label{Case B}	
\end{center}
\end{table}
}

{\footnotesize
\begin{table}[H]
\begin{center}
\begin{tabular}{|l|l|}
\hline \multicolumn{2}{|c|}{{\bf Case C of Proposition \ref{sigma5}}} \\
\hline $\;$&{\bf Discriminant $D(u,v)$} \\
	\hline
	\hline $\frac{2}{s_1-j}<\frac{1}{[\frac{s_1}{4}]+k}$& $\begin{array}{ll}D_1D_2D_3,\;\;\;m(D_l)=1,\;\;\;i_0(D_l,D_s)=\frac{4}{3}(s_1-j+[\frac{s_1}{4}]+k)& \hbox{\rm if} \gcd\{3,s_1-j+[\frac{s_1}{4}]+k\}=3\\
	D_1,\;\;\;\;\;\;\;\;\;\;\;\;S(D_1)=\langle 3, 4(s_1-j+[\frac{s_1}{4}]+k)\rangle&\hbox{\rm if} \gcd\{3,s_1-j+[\frac{s_1}{4}]+k\}=1\\
	\end{array}$\\ \hline
	\hline $\frac{2}{s_1-j}>\frac{1}{[\frac{s_1}{4}]+k}$& $\begin{array}{ll}D_1D_2D_3,\;\;\;m(D_l)=1,\;\;\;i_0(D_1,D_s)=2(s_1-j),\;\;i_0(D_2,D_3)=\frac{3}{2}(s_1-j)+[\frac{s_1}{4}]+k& \hbox{\rm if} \gcd\{2,s_1-j\}=2\\
	D_1D_2,\;\;\;\;\;\;\;m(D_1)=1\;\;\;\;S(D_2)=\langle2,3(s_1-j)+2([\frac{s_1}{4}]+k)\rangle\;\;\;i_0(D_1,D_2)=4(s_1-j)&\hbox{\rm if}
	 \gcd\{2,s_1-j\}=1\\
	\end{array}$\\
	\hline $\frac{2}{s_1-j}=\frac{1}{[\frac{s_1}{4}]+k}$& $\begin{array}{ll}D_1D_2D_3,\;\;\;m(D_l)=1,\;\;\;i_0(D_l,D_s)=2(s_1-j)& \hbox{\rm if } a_k\neq \pm \frac{4\sqrt{6}}{9}\\
	\hbox{\rm see Table } \ref{subC}&\hbox{\rm if } a_k=\pm \frac{4\sqrt{6}}{9}
	\\ 
	\end{array}$\\\hline
\end{tabular}
\caption{}
\label{Case C}	
\end{center}
\end{table}
}

{\tiny
\begin{table}[H]
\begin{center}
\begin{tabular}{|l|l|}
\hline $\;$&{\bf Discriminant $D(u,v)$} \\
	\hline
	\hline $s_1-2j>s$& $\begin{array}{ll}D_1D_2,\;\;\;\;\;\;\;m(D_1)=1,\;\;\;S(D_1)=\langle 2, 4(s_1-j)+3s\rangle,\;\;\;\;i_0(D_1,D_2)=4(s_1-j)&\hbox{\rm if } s \;\;\hbox{\rm is odd}\\ 
	D_1D_2D_3,\;\;\;m(D_l)=1,\;\;\;i_0(D_1,D_r)=2(s_1-j),\;\;\;\;\;\;\;\;\;\;\;\;i_0(D_2,D_3)=2(s_1-j)+\frac{3}{2}s &\hbox{\rm if } s\;\, \hbox{\rm is even}\\ 
	\end{array}$\\
	\hline
	\hline $s_1-2j<s$& $\;\;D_1D_2,\;\;\;\;\;\;\;m(D_1)=1,\;\;\;S(D_1)=\langle 2, 7s_1-10j\rangle,\;\;\;\;i_0(D_1,D_2)=4(s_1-j)$
	\\ \hline
	\hline $s_1-2j=s$ \;\;\hbox{\rm and } $a_{k+s}\neq \pm \frac{4\sqrt{6}}{81}$& $\begin{array}{ll}D_1D_2D_3,\;\;\;m(D_l)=1,\;\;\;i_0(D_l,D_r)=2(s_1-j)\;\; &\hbox{\rm if } \gcd(2,s)=2\\
	D_1D_2\;\;\;\;\;\;\;\;\;m(D_1)=1,\;\;\;S(D_2)=\langle 2, 7s_1-10j\rangle,\;\;\;i_0(D_1,D_2)=4(s_1-j)\;\; &\hbox{\rm if } \gcd(2,s)=1
	\end{array}$
	\\ 
	\\ \hline
	\hline $s_1-2j=s$ \;\;\hbox{\rm and } $a_{k+s}=\pm \frac{4\sqrt{6}}{81}$& $D_1D_2D_3,\;\;\;m(D_l)=1,\;\;\;i_0(D_1,D_r)=2(s_1-j),\;\;i_0(D_2,D_3)=4s_1-6j$.
	\\ 
\\ \hline
\end{tabular}
\caption{}
\label{subC}	
\end{center}
\end{table}
}

{\tiny
\begin{table}[H]
\begin{center}
\begin{tabular}{|l|l|}
\hline \multicolumn{2}{|c|}{{$\mathbf {r(C)=\mu(C)-\tau(C)\leq 2}$}} \\
	\hline {\bf Normal form} &{\bf Discriminant $D(u,v)$} \\
	\hline
	\hline $\begin{array}{l} r(C)=1\\S(f)=\langle s_0,s_1\rangle\end{array}$& $\begin{array}{ll}D_1^{s_0-3}D_2,\;\;\;\;\;D_1=(v-u^{s_1}),\;\;\;\;\;S(D_2)=\langle 3,s_1-2\rangle,\;\;\;i_0(D_1,D_2)=\min\left\{s_1,(s_1-2)s_0\right\}&\hbox{\rm if $s_1,s_0$ coprime} \\
	D_1^{s_0-3}D_2D_3\;\;D_1=(v-u^{s_1}),\;\;\;m(D_i)=1,\;\;i_0(D_1,D_k)=\min\left\{s_1,\frac{(s_1-2)s_0}{3}\right\}\;\; i_0(D_2,D_3)=\frac{(s_1-2)s_0}{3}&\hbox{\rm otherwise} 
	\end{array}$
 \\
 \hline
	\hline $\begin{array}{l} r(C)=2\\S(f)=\langle 4,s_1,s_2\rangle \end{array}$& $D_1D_2,\,\,\,\,\,\,m(D_1)=1,\,\,\,\,\,\,S(D_2)=\langle 2,s_2\rangle,\,\,\,\,\,\,  i_0(D_1,D_2)=2s_1$ \\
	 \hline
	\hline $\begin{array}{l} r(C)=2\\S(f)=\langle s_0,s_1\rangle \end{array}$& see Table \ref{TjurinaA} and Table \ref{TjurinaB} \\
	\hline 
\end{tabular}
\caption{Discriminants of branches with $r(C)=\mu(C)-\tau(C)\leq 2$}
\label{m2}	
\end{center}
\end{table}
}

{\tiny
\begin{table}[H]
\begin{center}
\begin{tabular}{|l|l|}
\hline \multicolumn{2}{|c|}{{\bf{Normal form:} $\mathbf{f(x,y)=y^{s_0}-x^{s_1}+x^{s_1-3}y^{s_0-2}},$ \bf{with} $\mathbf{2<s_0<s_1}$}} \\
	\hline $\;$ &{\bf Discriminant $D(u,v)$} \\
	\hline
	\hline $s_0$ odd, $s_1$ even& $D_1^{s_0-3}D_2,\;\;\;\;\;D_1=(v-u^{s_1}),\;\;\;\;\;S(D_2)=\langle 2,(s_1-3)s_0\rangle,\;\;\;i_0(D_1,D_2)=\min\left\{s_1,(s_1-3)s_0\right\}$
 \\
 \hline
	\hline otherwise& $D_1^{s_0-3}D_2D_3,\,\,\,\,\,\,m(D_i)=1,\,\,\,\,i_0(D_1,D_k)=\min\left\{s_1,\frac{(s_1-3)s_0}{2}\right\},\,\,\,\,\,\,\,\,  i_0(D_2,D_3)=\frac{(s_1-3)s_0}{2}$ \\
	 	\hline 
\end{tabular}
\caption{}
\label{TjurinaA}
\end{center}
\end{table}
}

{\tiny
\begin{table}[H]
\begin{center}
\begin{tabular}{|l|l|}
\hline \multicolumn{2}{|c|}{{\bf{Normal form:} $\mathbf{f(x,y)=y^{s_0}-x^{s_1}+x^{s_1-2}y^{s_0-3}+\left(\sum_{k \geq 2}^{2+[\frac{s_1}{s_0}]}a_{k}x^{s_1-k}\right)y^{s_0-2},}$  \bf{with} $\mathbf{4\leq s_0<s_1, \frac{2s_0}{s_0-3}<s_1}$ \bf{and} $\mathbf{a_k\in \C}$}} \\
	\hline $\;$ &{\bf Discriminant $D(u,v)$} \\
	\hline
	\hline $s_1-2$ (resp. $s_0$) and $3$ coprime & $D_1^{s_0-4}D_2,\;\;\;\;\;D_1=(v-u^{s_1}),\;\;\;\;\;S(D_2)=\langle 3,(s_1-2)s_0\rangle,\;\;\;i_0(D_1,D_2)=3s_1$
 \\
 \hline
	\hline otherwise& $D_1^{s_0-4}D_2D_3D_4,\,\,\,\,\,\,m(D_i)=1,\,\,\,\,i_0(D_1,D_k)=s_1,\,\,\,\,\,\,\,\,  i_0(D_l,D_r)=\frac{(s_1-2)s_0}{3},\;\;2\leq l\neq r\leq 4$ \\
	 	\hline 
\end{tabular}
\caption{}
\label{TjurinaB}
\end{center}
\end{table}
}


\begin{thebibliography}{999999999999}

\bibitem[B-Hef] {B-Hefez} V. Bayer, A. Hefez. {\em Algebroid plane curves whose Milnor and Tjurina numbers differ by one or two.} Bol. Soc. Brasil. Mat. (N.S.) 32 (2001), no. 1, 63-81.

\bibitem[Ca] {Casas-Asian} E. Casas-Alvero. {\em  Local geometry of planar analytic morphisms}. Asian J. Math. \textbf{11,} no. 3 (2007) 373-426.

\bibitem[Ch] {Chenciner} A. Chenciner. {\em Courbes alg\'ebriques planes}. Publications Math\'ematiques de l'Universit\'e Paris VII, 1978.

\bibitem[GB-Gwo]{GB-G} E. Garc\' \i a Barroso and J. Gwo\'zdziewicz. {\em A discriminant criterion of irreducibility}. Kodai Math. J., 35 (2) (2012), 403-414.

\bibitem[GB-Gw-L]{Hungarica} E. Garc\'{\i}a Barroso, J. Gwo\'zdziewicz and A. Lenarcik, 
{\em Non-degeneracy of the discriminant.} 
Acta Math. Hungar. Volume 147, Issue 1 (2015), 220-246. doi: 10.1007/s10474-015-0515-8.

\bibitem[GB-L-P]{GB-L-P2007} E. Garc\'{\i}a Barroso,  A. Lenarcik and A. P\l oski, {\em Characterization of non-degenerate plane curve singularities.} Univ. Iagel. Acta Math. No. 45 (2007), 276. 

\bibitem[Hef]{Hefez}{\sc A. Hefez}. {\it Irreducible Plane Curve Singularities}.  Sixth Worhshop at Sao Carlos. (2003), 1-120.


\bibitem[Hef-Her]{Hefez-Hernandes-2009} A. Hefez, M.E. Hernandes, {\em Analytic classification of plane branches up to multiplicity 4.} Journal of Symbolic Computation 44 (2009), 626-634.

\bibitem[Hef-Her-HI1] {Hefez-Hernandes-HI2017} A. Hefez; M. E. Hernandes; M.F. Hern\'andez Iglesias, {\em On Polars of Plane Branches} In: Cisneros-Molina J., Tr‡ng L D., Oka M., Snoussi J. (eds) Singularities in Geometry, Topology, Foliations and Dynamics. Trends in Mathematics. BirkhŠuser (2017), 135-153.

\bibitem[Hef-Her-HI2] {Hefez-Hernandes-HI2018} A. Hefez; M. E. Hernandes; M.F. Hern\'andez Iglesias, {\em Plane branches with Newton non-degenerate polars.} Internat. J. Math. 29 (2018), no. 1, 1850001, 12 pp.

\bibitem[HI]{tesis Fernando}  M.F. Hern\'andez Iglesias, {\em Polar de um germe de curva irredut\'{\i}vel plana.} PhD thesis. Universidade Federal Fluminense, Brasil (2012).

\bibitem[M]{Merle} M. Merle, {\em Invariants polaires des courbes planes.} Invent. Math., 41 (1977), 103-111.

\bibitem[O]{Oka} M. Oka, {\em Non-Degenerate  Complete  Intersection  Singularity},
Actualit\'es Math\'ematiques. Hermann, Paris, 1997,  viii+309 pp.

\bibitem[T]{Teissier} B.Teissier, {\em Variet\'es polaires.I. Invariants polaires des singularit\'es dÕhypersurfaces}, Invent. Math., 40 (1977), 267Ð292.

\bibitem[Z1]{Zariski-1966} O. Zariski. {\em Characterization of plane algebroid curves whose module of differentials has maximum torsion.} Proc. Nat. Acad. Sc. 56: (1966), 781-786.

\bibitem[Z2]{Zariski} O. Zariski, {\em The moduli problem for plane branches}, with an appendix by Bernard Teissier. University Lectures Series, Volume 39, AMS 2006,  pp. 
151.

\end{thebibliography}
\end{document}